\documentclass[12 pt]{amsart}
\usepackage{amsmath,amssymb,amsthm}
\usepackage{a4wide}
\usepackage{color}
\usepackage{hyperref}

\newcommand{\ZZ}{\mathbb{Z}}

\newcommand{\Aut}{{\rm{Aut}}}
\newcommand{\Sym}{{\rm{Sym}}}
\newcommand{\AGammaL}{{\rm{A}\Gamma\rm{L}}}

\newcommand{\A}{{\rm{A}}}
\newcommand{\V}{{\rm{V}}}
\newcommand{\E}{{\rm{E}}}
\newcommand{\K}{{\rm{K}}}
\newcommand{\Merge}{{\rm{M}}}
\newcommand{\PX}{{\rm{PX}}}
\newcommand{\dirPX}{\overrightarrow{\mathrm{PX}}}
\newcommand{\dirW}{\overrightarrow{\mathrm{W}}}
\newcommand{\vGa}{\Gamma}
\newcommand{\Dih}{\mathrm{Dih}}

\newcommand{\cM}{{\mathcal{M}}}
\newcommand{\cR}{{\mathcal{R}}}

\def\norm#1#2{{\bf N}_{#1}(#2)}
\def\cent#1#2{{\bf C}_{#1}(#2)}

\newtheorem{theorem}{Theorem}[section]

\newtheorem{lemma}[theorem]{Lemma}
\newtheorem{cor}[theorem]{Corollary}

\theoremstyle{definition}
\newtheorem{defn}[theorem]{Definition}
\newtheorem{notation}[theorem]{Notation}

\numberwithin{equation}{section}

\title{Elusive groups of automorphisms of digraphs of small valency}

\author{Michael Giudici, Luke Morgan, Primo\v{z} Poto\v{c}nik, Gabriel Verret}

\address{Michael Giudici and Luke Morgan: School of Mathematics and Statistics,\newline 
\indent  University of Western Australia, 35 Stirling Highway, Crawley, WA 6009, Australia.} 
\email{michael.giudici@uwa.edu.au, luke.morgan@uwa.edu.au} 

\address{Primo\v{z} Poto\v{c}nik: Faculty of Mathematics and Physics, University of Ljubljana, \newline
\indent Jadranska 21, SI-1000 Ljubljana, Slovenia.\newline
\indent Also affiliated with: IAM, University of Primorska, \newline
\indent Muzejski trg 2, SI-6000 Koper, Slovenia; and \newline
\indent Institute of Mathematics, Physics and Mechanics,\newline
\indent Jadranska 19, SI-1000 Ljubljana, Slovenia.
}
\email{primoz.potocnik@fmf.uni-lj.si}

\address{Gabriel Verret: School of Mathematics and Statistics, \newline
\indent University of Western Australia, 35 Stirling Highway, Crawley, WA 6009, Australia. \newline
\indent Also affiliated with: FAMNIT, University of Primorska, \newline
\indent Glagolja\v{s}ka 8, SI-6000 Koper, Slovenia.}
\email{gabriel.verret@uwa.edu.au}

\thanks{The research of the first and second authors is supported by the Australian Research Council grant DP120100446. The last author is supported by UWA as part of the ARC grant DE130101001.}

\begin{document}

\begin{abstract}
A transitive permutation group is called elusive if it contains no semiregular element. We show that no group of automorphisms of a connected graph of valency at most four is elusive and  determine all the elusive groups of automorphisms of connected digraphs of out-valency at most three. 
\end{abstract}

\maketitle

\section{Introduction}

A permutation is called \emph{semiregular} if all its cycles all have the same length.
Semiregular automorphisms play a special role in algebraic graph theory. For example, they have proved to be particularly helpful for constructing Hamiltonian cycles \cite{alspach}.  They have also been used to enumerate \cite{MR} and provide nice representations of vertex-transitive digraphs of small order \cite{biggs}. Moreover, many constructions of vertex-transitive digraphs yield obvious semiregular automorphisms, such as the Cayley construction and covering graph constructions.

Maru\v{s}i\v{c}  asked if every finite vertex-transitive digraph admits a semiregular automorphism \cite{Dragan}.  A digraph is a set together with an irreflexive but not necessarily symmetric binary relation. The word graph is reserved for the symmetric case.
It has been shown that all cubic \cite{DraganScap} and quartic \cite{Dobson} vertex-transitive graphs admit semiregular automorphisms. We significantly improve these results by proving the following, perhaps surprising, fact:

\begin{theorem}
\label{theorem:valency4}
Every vertex-transitive group of automorphisms of a connected graph of valency at most four contains a semiregular automorphism.
\end{theorem}

The corresponding result  for digraphs of out-valency three does not hold. However, counterexamples are rare and are completely classified in our next theorem.  

\begin{theorem}
\label{theorem:valency3Directed}
Let $\Gamma$ be a finite connected digraph of out-valency at most three and let $G\leqslant \Aut(\Gamma)$ be transitive on vertices. Then $G$ does not contain a semiregular element if and only if
$\vGa\cong\dirPX(3,2^a,1)$ and, up to conjugacy in $W$, $G= N\rtimes\Lambda_i\langle\rho\rangle$ for some $i\in\{1,\ldots,2^{a-2}\}$, where $W$, $N$, $\Lambda_i$ and $\rho$ are as in Notation~\ref{NotationDirected}.
\end{theorem}

Note that the group $G$ appearing in the conclusion of Theorem \ref{theorem:valency3Directed}  is a proper subgroup of $\Aut(\Gamma)$. We thus immediately obtain the following positive answer to Maru\v{s}i\v{c}'s question for digraphs of out-valency three:

\begin{cor}
\label{cor:Polyvalency3Directed}
Every finite vertex-transitive digraph of out-valency at most three admits a semiregular automorphism.
\end{cor}

The digraphs which occur in the conclusion of Theorem \ref{theorem:valency3Directed} will be defined in Section \ref{sec:PXU}. They are of the same flavour as the exceptions in \cite[Theorem 1.3]{GPV}, where the related question about edge-transitive groups of automorphisms with no semiregular element was considered.   
Theorems~\ref{theorem:valency4} and \ref{theorem:valency3Directed}  will be proved in Section~\ref{sec:last}, after some preliminary work in Sections~\ref{sec:prelim} and~\ref{sec:main}.

 Maru\v{s}i\v{c}'s question has been generalised to permutation groups. A transitive permutation group without a semiregular element is called \emph{elusive}. 
The Polycirculant Conjecture asserts that there is no  $2$-closed elusive group  \cite{sevenauthor}. The validity of this conjecture would positively answer Maru\v{s}i\v{c}'s question since the automorphism group of a digraph is itself $2$-closed. 
Given an elusive group it is then interesting to determine the connected digraphs upon which it acts and in particular those of smallest valency.  In this context, Theorem \ref{theorem:valency3Directed} gives a classification of all elusive groups with a connected suborbit of length three (that is, one for which the corresponding orbital digraph is connected) and Theorem \ref{theorem:valency4} says that there are no  elusive groups with a connected self-paired suborbit of length three or four.  This leads to the following natural question:

\medskip
\noindent
\textbf{Question.} \textit{What is the smallest integer $k$ such that there is an elusive group with a connected self-paired suborbit of length $k$?}
\medskip

A related question is to determine the smallest valency of a connected graph that admits an elusive group of automorphisms. By Theorem \ref{theorem:valency4} this is at least five.  We  note  that the group in Theorem \ref{PXuMain} (\ref{case3}) provides an example of an elusive group with a connected self-paired suborbit of length six. Hence the answer to either  of the above  questions is five or six.

\section{Preliminaries}\label{sec:prelim}
\subsection{Terminology}
A {\em digraph} $\vGa$ consists of a finite non-empty set of \emph{vertices} $\V(\vGa)$ and an  irreflexive binary relation $\A(\vGa)$ on $\V(\vGa)$. 
A digraph $\vGa$ is called {\em asymmetric} provided that $\A(\vGa)$  is asymmetric, while it is a \emph{graph} if $\A(\vGa)$ is symmetric. An element $(u,v)$ of $\A(\vGa)$ is called an \emph{arc} of $\vGa$; in this case, $v$ is called an \emph{out-neighbour} of $u$. The set of  out-neighbours of a vertex $v$ is denoted by $\vGa^+(v)$  and its cardinality is called  the \emph{out-valency} of $v$. If every vertex has the same out-valency  $d$, then $\Gamma$ is  said to have \emph{out-valency} $d$. If $\Gamma$ is a graph, then we sometimes write $\vGa(v)$ instead of $\vGa^+(v)$ and \emph{valency} instead of out-valency.

An \emph{automorphism} of a digraph $\vGa$ is a permutation of $\V(\vGa)$ that preserves $\A(\vGa)$. The set of automorphisms of $\Gamma$ forms a group, denoted by $\Aut(\Gamma)$. A digraph $\Gamma$ is said to be $G$-\emph{vertex-transitive} if $G$ is a subgroup of $\Aut(\Gamma)$ acting transitively on $\V(\Gamma)$. Similarly, $\Gamma$ is said to be $G$-\emph{arc-transitive} if $G$ acts transitively on $\A(\vGa)$. If $v$ is a vertex of $\Gamma$, $G_v^{\Gamma^+(v)}$ denotes the permutation group induced by the vertex-stabiliser $G_v$ in its action on $\Gamma^+(v)$.

Let $\Gamma$ be a $G$-vertex-transitive graph and let $N$ be a normal subgroup of $G$. The $N$-orbit containing a vertex $v$ is denoted by $v^N$. The \emph{normal quotient graph} $\Gamma/N$ has the $N$-orbits on $\V(\Gamma)$ as vertices, with  distinct vertices $u^N$ and $v^N$ being adjacent if and only if there exist $u' \in u^N$ and $v' \in v^N$ such that $u'$ and $v'$ are adjacent in $\Gamma$.  Note that $G$ has an induced transitive action on the vertices of $\Gamma/N$. Moreover, it is easily seen that the valency of $\Gamma/N$ is at most the number of orbits of $N_v^{\Gamma(v)}$.

We denote the dihedral group of order $2r$ in its natural faithful  action on $r$ points by $\Dih(r)$. Given a group $G$ and subgroups $X$ and $Y$, the normaliser (respectively, centraliser) of $X$ in $Y$ is denoted by $\norm Y X$ (respectively, $\cent Y X$).

If $p$ is a prime and $V$ is an elementary abelian $p$-group, then $V$ can be viewed naturally as a vector space over $\ZZ_p$. If $V$ is also normalised by some group $G$, then $V$ becomes a $\ZZ_p G$-module. We use this point of view a few times without comment.

\subsection{A few results}
We now prove a few results that will be useful later. Most of these are well-known  hence we only give a few proofs. 

\begin{lemma}
\label{lemma:divisors of g = divisors of gv}
Let $G$ be a transitive permutation group on $\Omega$. If a prime $p$ divides $|G|$ but $p$ does not divide $|G_\omega|$ for some $\omega \in \Omega$, then $G$ contains a semiregular element of order $p$.
\end{lemma}

\begin{lemma}
\label{lemma:Pgroup}
Every nontrivial central element of a transitive permutation group is semiregular.
\end{lemma}

\begin{lemma}\label{lemma:Normal Abelian Two Orbits}
Let $G$ be a transitive permutation group having a nontrivial abelian normal subgroup $N$ with at most two orbits. Then $N$ contains a semiregular element.
\end{lemma}
\begin{proof}
If $N$ is semiregular, then we are done. We may thus assume that $N$ is not semiregular and thus has exactly two orbits. Let $v$ and $w$ be representatives of these two orbits and let  $g\in G$ such that $v^g = w$. Let $x$ be an element of prime order in $N_v$. Note that $x$ fixes every element of $v^N$ and no element  of $u^N$. It easily follows that $xx^g$ is a semiregular element contained in $N$.
\end{proof}

\begin{lemma}
\label{lemma: divisors of gv = divisors of local}
Let $\vGa$ be a connected $G$-vertex-transitive digraph and let $v\in \V(\vGa)$. If $p$ is a prime dividing $|G_v|$, then $p$ divides $|G_v^{\vGa^+(v)}|$.
\end{lemma}
\begin{proof}
Since $\vGa$ is connected, vertex-transitive and finite, it follows easily that it is strongly connected. Let $g$ be an element of order $p$ in $G_v$. Since $g$ is nontrivial, it must move some vertex. Let $u$ be a vertex of $\vGa$ moved by $g$ at minimal distance from $v$. By the strong connectivity of $\vGa$, there is a path $v,v_1,\ldots,v_t,u$. By the minimality of $u$, we have $g\in G_{v_t}$. Since $g$ acts nontrivially on $\vGa^+(v_t)$, it follows that $p$ divides $|G_{v_t}^{\vGa^+(v_t)}|$ and, by vertex-transitivity, also $|G_v^{\vGa^+(v)}|$.
\end{proof}

\begin{cor}
\label{cor: local action a p group}
Let $\vGa$ be a connected $G$-vertex-transitive digraph and let $p$ be a prime. If $G_v^{\vGa^+(v)}$ is a $p$-group, then $G$ contains a semiregular element.
\end{cor}
\begin{proof}
By Lemma \ref{lemma: divisors of gv = divisors of local} we have that $G_v$ is a $p$-group. The result, then follows by Lemmas~\ref{lemma:divisors of g = divisors of gv} and~\ref{lemma:Pgroup}.
\end{proof}

\begin{lemma}\label{stupidcyclicmodules}
Let $\langle\sigma\rangle$ be a cyclic $2$-group, let $V$ be a $\ZZ_2 \langle\sigma\rangle$-module of dimension $n$ and let $\cent V \sigma$ be the submodule  consisting of elements fixed by $\sigma$. If $\dim \cent V \sigma=1$, then $V$ has exactly one submodule of each dimension $i\in\{0,1,\ldots,n\}$ and these submodules form a chain with respect to inclusion.
\end{lemma}
\begin{proof}
Since a $2$-group acting on a vector space over $\ZZ_2$ has a non-zero fixed-point, we can recursively construct a chain of submodules $V_0<V_1< \cdots <  V_n$ with $\dim V_i=i$.  It therefore suffices to show that  $V$ has a unique module of each possible dimension.  

We use induction on $n$. If $n\leqslant  2$, then the statement is obvious since  $\dim \cent V \sigma=1$. Assume now that $n\geqslant 3$ and suppose that there exist $U_1$ and $U_2$, two distinct submodules of dimension two. Since $U_1$ and $U_2$ are $\ZZ_2\langle\sigma\rangle$-modules they have nontrivial fixed-point spaces and hence $U_1 \cap U_2 = \cent V \sigma$. Let $\cent V \sigma=\langle v \rangle $, let $u\in U_1 \setminus \cent V \sigma$ and let $w\in U_2\setminus \cent V \sigma$. Then $u^\sigma = u+v$ and $w^\sigma = w+v$. This implies that $(u+w)^\sigma = u+v+w+v = u+w$ and $u+w\in\cent V \sigma=U_1 \cap U_2$, which is a contradiction. Hence $V$ has a unique submodule of dimension two. It follows that $\dim \cent{V/\cent V \sigma}{\sigma}=1$ and  we obtain the result by induction.
\end{proof}

\subsection{Praeger-Xu digraphs}\label{sec:PXU}

We now define the digraphs that appear in Theorem~\ref{theorem:valency3Directed}. (These digraphs were introduced and studied systematically in~\cite{PraegerDigraph,PraegerXu}.)

\begin{defn}\label{DefinitionDigraph}
Let $r$ and $s$ be integers such that $r\geqslant 3$ and  $1\leqslant  s \leqslant  r$. The \emph{Praeger-Xu digraph} $\dirPX(3,r,s)$ has vertex set  $\ZZ_3^s\times \ZZ_r$ and arc-set 
$$\{((x_1,x_2,\ldots,x_s,y),(x_2,\ldots,x_s,x_{s+1},y+1))\mid x_i\in \ZZ_3, y\in\ZZ_r\}.$$\hfill $\diamond$
\end{defn}

Here is another description of these digraphs that is sometimes easier to work with. The digraph $\dirPX(3, r, 1)$ is the lexicographic product of a directed cycle of length $r$ and an arcless graph on three vertices. In other words, $\V(\dirPX(3, r, 1)) = \ZZ_3 \times \ZZ_r$ with an arc from $(u, x)$ to $(v, y)$ if and only if $y - x=1$. For $s \geqslant 2$, the digraph $\dirPX(3, r, s)$ has vertex-set the set of directed paths of $\dirPX(3, r, 1)$ of length $s - 1$ and, given two such paths $P_1$ and $P_2$, there is an arc from $P_1$ to $P_2$ in $\dirPX(3, r, s)$ if and only if the initial $(s-2)$-subpath of $P_2$ is equal to the terminal $(s-2)$-subpath of $P_1$.  

The underlying graph of $\dirPX(3,r,s)$ is denoted by $\PX(3,r,s)$. It is not hard to see that $\dirPX(3,r,s)$ is a connected asymmetric digraph of order $r3^s$ with out-valency $3$ and that $\PX(3,r,s)$ is a connected graph of valency $6$.

There is a natural action of $\Sym(\ZZ_3) \wr \Dih(r)=\Sym(\ZZ_3)^r\rtimes\Dih(r)$ as a group of automorphisms of  $\PX(3,r,1)$, with an induced faithful action as a group of automorphisms of $\PX(3,r,s)$. Namely,  for $g=(g_0,\ldots,g_{r-1},h)\in \Sym(\ZZ_3)^r\rtimes\Dih(r)$ (with $g_0,\ldots,g_{r-1}\in\Sym(\ZZ_3)$ and $h\in\Dih(r)$), we have
\begin{equation}
(x_1,x_2,\ldots,x_s,y)^g=(x_1^{g_y},x_2^{g_{y+1}},\ldots,x_s^{g_{y+s-1}},y^h), \label{gogo}
\end{equation}
where the indices of $g$ are taken modulo $r$. This action of  $\Sym(\ZZ_3) \wr \Dih(r)$ on $\ZZ_3^s\times \ZZ_r$  will play a crucial role in our investigation and is the main topic of the next section.

\section{Some elusive subgroups of $\Sym(\ZZ_3) \wr \Dih(r)$}\label{sec:main}

We first fix some notation. 

\begin{notation}\label{NotationDirected}
Let $r\geqslant 3$, let $\Omega=\{0\}\times \ZZ_r\subseteq \ZZ_3\times \ZZ_r$ and  let $W=\Sym(\ZZ_3)\wr \Dih(r)=\Sym(\ZZ_3)^r\rtimes\Dih(r)$. We view $W$  as an imprimitive permutation group on $\ZZ_3\times\ZZ_r$ in the natural way. Let $\sigma$ and $\tau$ be the elements of $W$ such that, for $(x,y)\in\ZZ_3\times\ZZ_r$ we have:
\begin{eqnarray*}
(x,y)^\sigma &=& (x,y+1),\\
(x,y)^\tau &=& (x,-y-1).
\end{eqnarray*}
Let $D=\langle\sigma,\tau\rangle$, let $\phi$ be the natural epimorphism from $W$ onto $D$, let $B=\ker \phi= \Sym(\ZZ_3)^r$ and let 
$\dirW=B\rtimes \langle\sigma\rangle$. Let $S$ be the unique Sylow $3$-subgroup of $B$. 

To reduce ambiguity, we denote by $\bf{1}$ the identity of $\Sym(\ZZ_3)$ and by $\bf{-1}$ the involution of $\Sym(\ZZ_3)$ that fixes $0$. We also think of $\ZZ_3$ as acting regularly on itself and thus $\Sym(\ZZ_3)=\ZZ_3\rtimes\langle \bf{-1}\rangle$.  Let $V=\langle {\bf{-1}} \rangle^r\leqslant B$.

We write elements of $B$ as $r$-tuples of elements of $\Sym(\ZZ_3)$ indexed by $(0,\ldots,r-1)$. For $x\in\Sym(\ZZ_3)$ and $i\geqslant 1$, we denote by $[x]^i$ the $i$-tuple $(x,x,\ldots,x)$. For example, $[{\bf{1}}]^r$ is the identity of $B$. Set
 $$\alpha=({\bf{1}},{\bf{1}},\ldots,{\bf{1}},{\bf{1}},{\bf{-1}})\in V\>\>\textrm{ and }\>  \rho=\alpha\sigma.$$
If $r$ is  a power of $2$ (and thus $r\geqslant 4$) then let 
 $$\beta=([{\bf{1}}]^{r/4},[{\bf{-1}}]^{r/4},[{\bf{1}}]^{r/4},[{\bf{-1}}]^{r/4})\in V,\>\>\mu=\beta\tau,$$
and define the following subgroups of $B$:
\begin{eqnarray*}
 N  & = & \{(x+y,x,y,x-y)\mid x,y\in \mathbb{Z}_3^{r/4}\}, \\
 N^*  & =& \{  (-x+y,x,y,x+y)\mid x,y\in \mathbb{Z}_3^{r/4}\}, \\
 \Lambda & =& \{  (z,z,z,z)  \mid z\in \langle{\bf{-1}}\rangle^{r/4}\}.
\end{eqnarray*}

Note that $\Lambda$ is normalised by $\langle\sigma\rangle$ and thus can be viewed as a $\ZZ_2 \langle\sigma\rangle$-module. Observe that $\cent \Lambda \sigma$ is the $1$-dimensional subspace of $\Lambda$ spanned by $[{\bf{-1}}]^r$. By Lemma~\ref{stupidcyclicmodules}, $\Lambda$ has a unique submodule of each dimension $i\in\{0,1,\ldots,r/4\}$, which we will denote by $\Lambda_i$.  \hfill $\diamond$
\end{notation}

Note that $\langle \rho \rangle $ has two orbits on $\ZZ_3\times\ZZ_r$, one of which is $\Omega$. In particular, $\rho$ has order $2r$. Observe also that $\dirW$ acts as a group of automorphisms of $\dirPX(3,r,s)$.  We now prove a series of technical lemmas using Notation~\ref{NotationDirected}.

\begin{lemma}\label{LemmaTechnical2}
Assume Notation~\ref{NotationDirected}. If $r$ is a power of $2$ and $\lambda\in\Lambda$, then $\langle \rho, \lambda\mu\rangle=\Lambda\langle\rho,\mu\rangle$ and, in particular, $\Lambda\leqslant \langle\rho,\mu\rangle$. 
\end{lemma}
\begin{proof}
Recall that $r\geqslant3$ and thus $r\geqslant4$.  Observe that, since $V$ is abelian, $\rho$ and $\mu$ act on $V$ as $\sigma$ and $\tau$, respectively. Note  that $\langle \rho,\mu\rangle$ normalises $\Lambda$ and hence $\langle \rho, \lambda\mu\rangle\leqslant\langle \Lambda,\rho,\mu\rangle=\Lambda\langle\rho,\mu\rangle$.  It thus suffices to show that $\Lambda\leqslant  \langle \rho, \lambda\mu\rangle$. By definition, we have $\lambda=(z,z,z,z)$ for some  $z\in \langle{\bf{-1}}\rangle^{r/4}$. Let  $x=\lambda\beta$, that is, $x=(z,{\bf{-1}} z,z,{\bf{-1}}z)$. Since $V\cong \ZZ_2^r$ and since $\sigma$ normalises $V$, we may think of $V$ as a $\ZZ_2\langle \sigma \rangle$-module.  Let $U$  be the  submodule of $V$ generated by $x$. Note that $\Lambda$ is a submodule of $V$ but $x\notin \Lambda$ and thus $U\nleqslant\Lambda$. Since $\cent V {\sigma} =\langle ({\bf{-1}},\dots,{\bf{-1}})\rangle$, we may apply Lemma~\ref{stupidcyclicmodules}, which yields that  $U>\Lambda$ and thus  $\dim U> \dim \Lambda=r/4$. 
Now $\cent U {\sigma}=\langle ({\bf{-1}},\dots,{\bf{-1}})\rangle$ is 1-dimensional and $U/ \mathbf C_U(\sigma) \cong [U,\langle\sigma\rangle]$, therefore the commutator  subspace  $[U,\langle\sigma\rangle]$ has codimension 1 in $U$.
 It follows that $\dim[U,\langle\sigma\rangle]\geqslant r/4$ and hence, by Lemma~\ref{stupidcyclicmodules}, $\Lambda\leqslant  [U,\langle\sigma\rangle]$. We now show that $[U,\langle\sigma\rangle] \leqslant  \langle \rho, \lambda\mu\rangle$.

Observe first that the group $[U,\langle\sigma\rangle]$ is generated by all the elements of the form $x^{\sigma^i}x^{\sigma^j}$ for  $i$ and $j$ integers. Since $\tau$ normalises $V$ and $V$ is abelian, we see that $x^{\sigma^i}x^{\sigma^j} =x^{\sigma^i}x^\tau x^{\sigma^j} x^\tau$. In order to show that $[U,\langle\sigma\rangle] \leqslant  \langle \rho, \lambda\mu\rangle$, it thus suffices to show that $x^{\sigma^i}x^\tau \in \langle \rho, \lambda\mu\rangle$ for every integer $i$.

It follows immediately from Notation~\ref{NotationDirected} that $\alpha^2=1$, $\alpha^\tau = \alpha^\sigma$ and $\sigma^\tau = \sigma^{-1}$, implying that
\begin{equation}
\rho^\tau = (\alpha\sigma)^\tau = \alpha^\tau \sigma^\tau = \alpha^\sigma \sigma^{-1} = \sigma^{-1}\alpha = (\alpha\sigma)^{-1} = \rho^{-1}. \label{gogo2}
\end{equation}
This implies that $\tau = \rho^{-i}\tau\rho^{-i}$ for every integer $i$. Hence
$$x^{\sigma^i}x^\tau = x^{\sigma^i}\tau x \tau = x^{\sigma^i} \rho^{-i}\tau\rho^{-i} x \tau = \rho^{-i} x^{\sigma^i\rho^{-i}} \tau\rho^{-i} x \tau.$$
Since $\rho = \alpha\sigma$, $\alpha\in V$ and $\sigma$ normalises $V$, we see that $\rho^i = \gamma \sigma^i$ for some $\gamma\in V$. As $V$ is abelian, it follows that $x^{\sigma^i\rho^{-i}} = x^\gamma=x$ and hence
$$x^{\sigma^i}x^\tau=\rho^{-i} x \tau\rho^{-i} x \tau=(\rho^{-i} x \tau)^2=(\rho^{-i} \lambda\beta\tau)^2 =(\rho^{-i}\lambda\mu)^2 \in  \langle \rho, \lambda\mu\rangle.$$ 
This implies that $[U,\langle\sigma\rangle] \leqslant  \langle \rho, \lambda\mu\rangle$, completing the proof that $\langle \rho, \lambda\mu\rangle=\Lambda\langle\rho,\mu\rangle$. Taking $\lambda=1$ yields $\Lambda\leqslant \langle\rho,\mu\rangle$. 
\end{proof}

\begin{lemma}\label{lemmaNormalisers}
Assume Notation~\ref{NotationDirected}. If $r$ is a power of $2$, then 
\begin{enumerate}
\item $\norm{V}{N}=\langle\rho,\mu\rangle\cap B= \Lambda$ and \label{oneone}
\item $\norm {V\rtimes D} {N}= \langle\rho,\mu\rangle$.
\end{enumerate}
\end{lemma}
\begin{proof}
By Lemma~\ref{LemmaTechnical2}, we have $\Lambda\leqslant \langle\rho,\mu\rangle\cap B$. Note that $V$ is the stabiliser of $\Omega$ in $B$. Since $\langle\rho,\mu\rangle$ stabilises $\Omega$, it follows that $\langle\rho,\mu\rangle\cap B\leqslant V$. It is easy to check that $\langle\rho,\mu\rangle$ normalises $N$ hence $\langle\rho,\mu\rangle\cap B\leqslant \norm{V}{N}$. We now show that $\norm{V}{N}\leqslant\Lambda$.

 Let $g=(g_0,\ldots,g_{r-1})\in \norm{V}{N}$. For $i\in\{0,\ldots,r/4-1\}$, let $b_i$ be the $i^{\mathrm{th}}$ element of the standard basis for $\ZZ_3^{r/4}$. Note that $(b_i,b_i,0,b_i)\in N$ and thus $(b_i,b_i,0,b_i)^g\in N$. This yields $g_i=g_{i+r/4}=g_{i+3r/4}$ for every $i\in\{0,\ldots,r/4-1\}$. A similar argument using $(b_i,0,b_i,-b_i)$ instead of $(b_i,b_i,0,b_i)$ yields that $g_i=g_{i+r/2}$.   Since this holds for every $i\in\{0,\ldots,r/4-1\}$, it follows that $g\in\Lambda$. This shows that  $\norm{V}{N}\leqslant\Lambda$ and  concludes the proof of~(\ref{oneone}).

Note that $V\rtimes D=V \langle \rho,\mu\rangle$. Since $\langle \rho,\mu\rangle$ normalises $N$, the following holds (Dedekind's modular law is used in the second equality)
 $$\norm {V\rtimes D} {N}=V \langle \rho,\mu\rangle \cap \norm {V\rtimes D} {N}=(V \cap \norm{V\rtimes D}{N})\langle \rho,\mu\rangle=\norm{V}{N}\langle \rho,\mu\rangle=\langle \rho,\mu\rangle.$$
\end{proof}

\begin{lemma}\label{LemmaTechnical}
Assume Notation~\ref{NotationDirected}. If $r$ is a power of $2$, then $N$ and $N^*$ are the only nontrivial proper subgroups of $S$ normalised by $\rho$. Moreover, $N^\tau=N^*$.
\end{lemma}

\begin{proof}
It is easy to check that $N^\tau=N^*$. We have already noted that $N$ is a nontrivial proper subgroup of $S$ normalised by $\rho$. Since $\rho^\tau=\rho^{-1}$, the same holds for $N^*$. It remains to show that $N$ and $N^*$ are the only such subgroups.

Since $S\cong\ZZ_3^r$, we view $S$ as a vector space over $\ZZ_3$. The action of $\rho$ on $S$ by conjugation induces a linear transformation of $S$ whose corresponding matrix (with respect to the natural basis)  is:
$$A=\begin{pmatrix}
  0 & 1   & 0&\cdots & 0 \\
  \vdots & \ddots    & \ddots & \ddots& \vdots\\
  \vdots  &     & \ddots & \ddots&0  \\
  0 &  &  & \ddots&1\\
    -1 & 0 & \cdots &\cdots& 0
 \end{pmatrix}.$$

Observe that the characteristic polynomial of $A$ is $x^r+1$. Note that $S$ is a cyclic $\langle\rho\rangle$-module and hence $x^r+1$ is also the minimal polynomial of $A$ by \cite[Theorem 2.1]{neumannpraeger}.

Since $r$ is a power of $2$, $x^r+1$ is the $(2r)^{\rm th}$ cyclotomic polynomial (see \cite[Example 2.46]{FF} for example). We leave it to the reader to check that the smallest positive integer $d$ such that $3^d \equiv 1\pmod {2r}$ is $d=r/2$. It follows that  $x^r+1$ factors over $\ZZ_3$ as a product of two irreducible polynomials each having degree $r/2$ (see \cite[Theorem 2.47(ii)]{FF}). Since $\ZZ_3$ is a field, this factorisation is unique and thus $x^r+1$ has exactly two nontrivial proper factors.

Since $S$ is a cyclic $\langle\rho\rangle$-module, there is a bijection between the set of subgroups of $S$ normalised by $\rho$ and the set of factors of the minimal polynomial of $A$, that is, the set of factors of $x^r+1$. From the above paragraph we conclude that $S$ has exactly two nontrivial proper subgroups normalised by $\rho$, namely $N$ and $N^*$.
\end{proof}

\begin{theorem}\label{PXuMain}
Assume Notation~\ref{NotationDirected}, let $1\leqslant  s\leqslant  r-1$ and let $W$ act on $\ZZ_3^s\times\ZZ_r$ as in  $(\ref{gogo})$. If $G$ is a subgroup of $W$ such that $\sigma\in \phi(G)$, then, up to conjugacy in $W$, exactly one of the following occurs: 
\begin{enumerate}
\item $G$ contains a semiregular element;
\item  $(r,s)=(2^a,1)$ and  $G= N\rtimes\Lambda_i\langle\rho\rangle$ for some $i\in\{1,\ldots,r/4\}$;  \label{case2}
\item $(r,s)=(4,1)$ and  $G=N\rtimes\langle\rho,\mu\rangle$.  \label{case3}
\end{enumerate}
\end{theorem}
\begin{proof}
We first show that at least one of (1), (2) or (3) occurs.  Let $g$ be an element of $G$ such that $\phi(g)=\sigma$. Clearly, the orbits of $g$ on $\ZZ_3\times\ZZ_r$ have length divisible by $r$. If $g$ has three orbits of length $r$ or one orbit of length $3r$, then it is semiregular. We thus assume that $g$ has two orbits, with lengths $r$ and $2r$. In particular, $g$ has order $2r$.

If $r$ is divisible by some odd prime $p$, then $g^{2r/p}$ is a fixed-point-free element of order $p$ and hence is semiregular. We may thus assume that $r$ is a power of $2$. Since $W=(S\rtimes V)\rtimes D$, it follows that $V\rtimes D$ is a Sylow $2$-subgroup of $W$ and that $S \cap G$ is the unique Sylow $3$-subgroup of $G$. Let $P$ be a Sylow $2$-subgroup of $G$ containing $g$. By replacing $P$ and $g$ by some suitable $S$-conjugates, we may assume that $P\leqslant V\rtimes D$ and thus $g$ stabilises $\Omega$. Since $g$ has order $2r$ and $\phi(g)=\sigma$, it can be seen that, up to conjugation by $V$ we may assume that $g=\rho$.

If $S\cap G=1$, then $G$ is a $2$-group and the result follows by Lemma~\ref{lemma:Pgroup}. If $S\cap G=S$, then $G$ contains a semiregular element (for example $[1]^r$). We may thus assume that $1<S\cap G<S$. Since $\rho\in  G$, it follows by Lemma~\ref{LemmaTechnical} that $S\cap G\in\{N,N^*\}$. By Lemma~\ref{LemmaTechnical}, $\tau$ interchanges $N$ and $N^*$. Moreover,  $\tau$ normalises both $\langle \rho \rangle$ (see~(\ref{gogo2})) and $V\rtimes D$ hence $\rho\in P^\tau\leqslant V\rtimes D$. In particular, up to conjugation by $\langle\tau\rangle$ and replacing $P$ by $P^\tau$ if necessary, we have $S\cap G=N$. By Lemma~\ref{lemmaNormalisers} (2), it follows that $P\leqslant\langle \rho,\mu\rangle$.

As can be seen directly from~(\ref{gogo}), an element of $S$ fixing an element of $\ZZ_3^s\times\ZZ_r$ must have at least $s$ consecutive coordinates with value $0$. If $r=4$, then let $n=(-1,1,1,0)\in N$. If $r\geqslant 8$, then let $u=[1]^{r/4}$, let $v=[1,-1]^{r/8}$ and let $n=(u+v,u,v,u-v)\in N$. It is easy to check that $n$ does not have two consecutive coordinates with value $0$ and thus is semiregular if $s\geqslant 2$. We may thus assume that $s=1$.

Suppose that $G\leqslant \dirW$. By Lemma~\ref{lemmaNormalisers}~(\ref{oneone}), we have $B\cap \langle \rho,\mu\rangle=\Lambda$.
 Therefore $\langle \rho,\mu\rangle\cap \dirW=\Lambda\langle\rho\rangle$ and hence $P\leqslant \Lambda\langle\rho\rangle$.  Since $\rho\in P$, it follows that $P=(\Lambda\cap P)\langle\rho\rangle$. Note that $\Lambda\cap P$ is a $\langle\rho\rangle$-invariant and thus $\langle\sigma\rangle$-invariant subgroup of $\Lambda$ and hence  is equal to $\Lambda_i$ for some $i\in\{0,1,\ldots,r/4\}$ (recall Notation~\ref{NotationDirected}). In particular, $G=N\rtimes P=N\rtimes \Lambda_i\langle\rho\rangle$. Since $\Lambda_1=\langle \bf[{-1}]^r\rangle\leqslant\langle\rho\rangle$ it follows that $i\geqslant 1$ and hence (\ref{case2}) holds. 

We now assume that $G\nleqslant \dirW$. Since $\sigma\in\phi(G)$, it follows that $\phi(G)=D$. In particular, there exists $d\in G$ such that $\phi(d)=\phi(\mu)$. Since $G\leqslant N\rtimes\langle \rho,\mu\rangle$ and $B\cap \langle \rho,\mu\rangle=\Lambda$, it follows that  $d=n\lambda\mu$ for some $n\in N$ and $\lambda\in\Lambda$. As $N\leqslant G$, we have that $\lambda\mu\in G$.  By Lemma~\ref{LemmaTechnical2}, $G=N\rtimes\langle \rho,\mu\rangle$.  If $r=4$, then (\ref{case3}) holds. We thus assume that $r\geqslant 8$. Let $t=([{\bf{1}}]^{r/8},[{\bf{-1}}]^{r/8})$ and let $h=(t,t,t,t)\mu$.  Note that $(t,t,t,t)\in \Lambda$ and thus by Lemma~\ref{LemmaTechnical2} $h\in G$. Moreover, $\phi(h)=\tau$ hence $h$ is fixed-point-free. Finally, an easy calculation shows that $h$ has order $2$ hence it is semiregular.  This concludes the proof that at least one of (1), (2) or (3) occurs. 

It remains to show that, in cases (\ref{case2}) and (\ref{case3}), $G$ does not contain a semiregular element. Let $n=(n_0,\ldots,n_{r-1})\in N$, with $n_i\in\ZZ_3$. By definition, we have $n_0=n_{r/2}+n_{r/4}$ and $n_{3r/4}=n_{r/4}-n_{r/2}$. This implies that at least one of $\{n_0,n_{r/4},n_{r/2},n_{3r/4}\}$ is equal to $0$. Since $s=1$, it follows that $n$ is not semiregular. It thus remains to show that $G$ does not contain a semiregular involution and, by Sylow's theorems, it suffices to check involutions in a single Sylow $2$-subgroup.

 In case (\ref{case2}), we consider the Sylow $2$-subgroup $\Lambda\langle \rho\rangle$.   Note that every involution of $\Lambda\langle \rho\rangle$ is either in $\Lambda$ or in the coset $\Lambda \rho^{r/2}$. An easy calculation shows that the latter coset does not contain an involution, while $\Lambda$ does not contain a semiregular element.

In case (\ref{case3}), we consider the Sylow $2$-subgroup $\langle\rho,\mu\rangle$. Since $r=4$, we have that $\rho\mu$ is an involution and that $\rho^{\mu}=\rho^3$. It follows that $\langle\rho,\mu\rangle$ is a semidihedral group of order $16$, with two conjugacy classes of involutions, represented by $\rho^4=(\bf{-1},\bf{-1},\bf{-1},\bf{-1})$ and $\rho\mu$. Both these elements have fixed points, which concludes the proof.
\end{proof}

\noindent\textbf{Remark.} \textit{Note that in Theorem~\ref{PXuMain}, if we consider conjugacy by $\dirW$ rather than by $W$ then cases $(2)$ and $(3)$ split into two cases each, one featuring $N$ and the other $N^*$.}

\section{Proof of Theorems~\ref{theorem:valency3Directed} and~\ref{theorem:valency4}} \label{sec:last}

We now prove Theorems~\ref{theorem:valency3Directed} and~\ref{theorem:valency4}, which we restate for convenience.

\medskip

\noindent\textbf{Theorem~\ref{theorem:valency3Directed}.}\emph{
A connected digraph $\vGa$ of out-valency at most three admits an elusive group of automorphisms $G$ if and only if $\vGa\cong\dirPX(3,2^a,1)$ and, up to conjugacy in $W$, $G= N\rtimes\Lambda_i\langle\rho\rangle$ for some $i\in\{1,\ldots,2^{a-2}\}$, where $W$, $N$, $\Lambda_i$ and $\rho$ are as in Notation~\ref{NotationDirected}.
}
\begin{proof}
Let $\vGa$ be a connected digraph of out-valency at most three that admits  an elusive group $G$ of automorphisms.
Let $v\in \vGa$.  By Corollary~\ref{cor: local action a p group}, we may assume that $G_v^{\vGa^+(v)}$ is not a $p$-group hence $|\vGa^+(v)|=3$ and $G_v^{\vGa^+(v)}\cong\mathrm{Sym}(3)$. In particular, $G_v^{\vGa^+(v)}$ is transitive and hence $\vGa$ is $G$-arc-transitive. By Lemmas~\ref{lemma:divisors of g = divisors of gv} and \ref{lemma: divisors of gv = divisors of local} we may assume that $G$ is a $\{2,3\}$-group and therefore soluble. 

Let $N$ be a minimal normal subgroup of $G$. Note that $N$ is abelian. We may assume that $N$ is not semiregular hence $N_v \neq 1$ and $N_v^{\vGa^+(v)}\neq 1$.  Since $N_v^{\vGa^+(v)}$ is a nontrivial normal subgroup of $G_v^{\vGa^+(v)}$ which is primitive, it follows that $N_v^{\vGa^+(v)}$ is transitive. By Lemma~\ref{lemma:Normal Abelian Two Orbits}, we may also assume that $N$ has at least three orbits, this implies that $\vGa$ is not a graph and, since  $\Gamma$ is $G$-arc-transitive,  it must be an asymmetric digraph. It then follows by~\cite[Theorem~$2.9$]{PraegerDigraph} that $\vGa\cong\dirPX(3,r,s)$ for some $r\geqslant 3$ and $1\leqslant  s\leqslant  r-1$. Assume Notation~\ref{NotationDirected}. By~\cite[Theorem 2.8]{PraegerDigraph}, we have that $G\leqslant \dirW$ and, since $\vGa$ is $G$-arc-transitive,  $\sigma\in \phi(G)$.  It follows by Theorem~\ref{PXuMain} that $\vGa\cong\dirPX(3,2^a,1)$ and that, up to conjugacy in $W$, $G= N\rtimes\Lambda_i\langle\rho\rangle$ for some $i\in\{1,\ldots,2^{a-2}\}$.

The converse also follows by Theorem~\ref{PXuMain}, that is, if $\vGa\cong\dirPX(3,2^a,1)$ and $G= N\rtimes\Lambda_i\langle\rho\rangle$, then $G$ is elusive.
\end{proof}

\noindent\textbf{Theorem~\ref{theorem:valency4}.}\emph{ There is no elusive group of automorphisms of a connected graph of valency at most four.}
\begin{proof}
Let $\Gamma$ be a connected graph of valency at most four with an elusive group $G$ of automorphisms. By Theorem \ref{theorem:valency3Directed} we may assume that $\Gamma$ has valency four. Let $v\in \Gamma$.  By Corollary~\ref{cor: local action a p group}, we may assume that $G_v^{\vGa(v)}$ is not a $p$-group. As it is a permutation group of degree $4$, the only possibilities are that it is isomorphic to one of $\mathrm{Sym}(3)$, $\mathrm{Alt}(4)$ or $\mathrm{Sym}(4)$. By Lemmas~\ref{lemma:divisors of g = divisors of gv} and \ref{lemma: divisors of gv = divisors of local} we may assume that $G$ is a $\{2,3\}$-group and therefore soluble. Let $N$ be a minimal normal subgroup of $G$. We may assume that $N$ is not semiregular hence $N_v \neq 1$ and $N_v^{\Gamma(v)}\neq 1$. By Lemma~\ref{lemma:Normal Abelian Two Orbits}, we may also assume that $N$ has at least three orbits.

Suppose first that $G_v^{\Gamma(v)}$  is primitive. Since $N_v^{\Gamma(v)}$ is a nontrivial normal subgroup of $G_v^{\Gamma(v)}$, it follows that $N_v^{\Gamma(v)}$ is transitive and hence $N$ has at most two orbits, which is a contradiction.

We may thus assume that $G_v^{\Gamma(v)}$ is not primitive and hence is isomorphic to $\mathrm{Sym}(3)$. In particular, $N$ is an elementary abelian $3$-group, $N_v^{\Gamma(v)}$ has the same orbits as $G_v^{\Gamma(v)}$, and every vertex $u$ has a unique neighbour $u'$ such that $G_u$ fixes $u'$. It is easily seen that $G_u=G_{u'}$ and $(u')'=u$. It follows that $\cM:=\{ \{u,u'\} \mid u\in\V(\Gamma)\}$ is a $G$-invariant perfect matching. Let $\cR=\E(\Gamma)\setminus\cM$. Note that $\cR$ is also a $G$-edge-orbit.

Since $N_v^{\Gamma(v)}$ has two orbits, $\Gamma/N$ has valency at most two. Since $N$ has at least three orbits, $\Gamma/N$ is a cycle and its edges form two orbits under the action of $G$, corresponding to $\cM$ and $\cR$. In particular, $\Gamma/N$ has even order. 

Suppose that $\Gamma/N$ has order $4$ and let $\{u,v\}\in\cR$. Then $\{u,u',v,v'\}$ is a set of representatives for the $N$-orbits. Let $\Omega=u^N\cup v^N$. Note that $\Omega$ is a block for the action of $G$, that $G^\Omega$ is transitive, that $N^\Omega$ is nontrivial and has two orbits. It follows by Lemma~\ref{lemma:Normal Abelian Two Orbits} that $N^\Omega$ contains a semiregular element. Clearly, if $n^\Omega\in N^\Omega$ is a semiregular element, then so is $n$ since $\V(\Gamma) \setminus \Omega=(u')^N\cup (v')^N$. This is a contradiction.

We may thus assume that $\Gamma/N$ has order at least $6$. In particular, between any two edges of $\cM$, there is at most one edge in $\Gamma$. We define a new graph $\Merge\Gamma$, with vertex-set $\cM$ and two elements $\{u,u'\}$ and $\{w,w'\}$ of $\cM$ adjacent if and only if there is an edge in $\Gamma$ between  $\{u,u'\}$ and $\{w,w'\}$. It is not hard to check that $\Merge\Gamma$ is a connected $6$-valent graph on which $G$ acts faithfully and arc-transitively. 

If an element $g$ of $G$ fixes a vertex $w$ of $\Gamma$ then $g$ must also fix $w'$ and thus $g$ fixes the vertex $\{w,w'\}$ of $\Merge \Gamma$. We  may thus assume that $G$ is elusive on $\Merge \Gamma$ and hence that $N$ is not semiregular on $\Merge \Gamma$. By~\cite[Theorem~$1$]{PraegerXu}, it follows that $\mathrm M \Gamma\cong\mathrm{PX}(3,r,s)$ for some $r\geqslant 3$ and $1\leqslant  s \leqslant  r-1$.

Assume Notation~\ref{NotationDirected}. If $G\leqslant  W$ then, since $\Merge \Gamma$ is $G$-arc-transitive, $\phi(G)=D$ and Theorem~\ref{PXuMain} implies that $(r,s)=(4,1)$. Similarly, if $G\nleqslant W$ then~\cite[Theorem 2.13]{PraegerXu} again implies that  $(r,s)=(4,1)$. In particular, $\Merge\Gamma$ is isomorphic to $\PX(3,4,1)$ which is isomorphic to $\K_{6,6}$, a complete bipartite graph of order $12$. Since $G$ is a $\{2,3\}$-group and acts arc-transitively on $\K_{6,6}$, the classification of elusive groups of degree $12$ (see~\cite[Figure 1]{sevenauthor}) yields that $G\cong\AGammaL(1,9)$. However, it can be checked that this group does not have an elusive action of degree $|\V(\Gamma)|=2|\V(\Merge\Gamma)|=24$.
\end{proof}

{\bf Acknowledgement.} The third author would like to thank the other three for their hospitality during his visit to Perth in December 2013. We are grateful to the referees for useful comments and feedback.

\end{document}